\documentclass[twoside,a4paper,leqno]{article}

\usepackage[utf8]{inputenc}

\usepackage [a4paper, top=2.5cm, bottom=2.5cm, left=2cm, right=2cm, bindingoffset=8mm] {geometry}

\usepackage{enumerate}

\usepackage{amsmath,amssymb,amsfonts,amsthm,xy}
\input xy \xyoption{all}

\usepackage[usenames,dvipsnames]{color}

\swapnumbers

\numberwithin{equation}{section}

\makeatletter
\def\tagform@#1{\maketag@@@{\ignorespaces#1\unskip\@@italiccorr}}
\makeatother

\theoremstyle{definition}
\newtheorem{Def}[equation]{Definition}

\theoremstyle{definition}
\newtheorem{Rmk}[equation]{Remark}

\theoremstyle{plain}
\newtheorem{Lemma}[equation]{Lemma}

\theoremstyle{plain}
\newtheorem{Prop}[equation]{Proposition}

\theoremstyle{plain}
\newtheorem{Thm}[equation]{Theorem}

\theoremstyle{plain}

\newcommand{\N}{\mathbb{N}}
\newcommand{\Z}{\mathbb{Z}}
\newcommand{\R}{\mathbb{R}}
\newcommand{\C}{\mathbb{C}}

\newcommand{\bb}{\mathcal{B}}
\newcommand{\w}{\mathsf w}

\newcommand{\new}[1]{#1}

\usepackage{framed}

\usepackage{ulem}
\title{Intrinsic Localization of Anisotropic Frames II: $\alpha$-Molecules}
\author{Philipp Grohs and Stefano Vigogna\footnote{This work has been carried out during a 2 month visit of the second named author at ETH Z\"urich in 2013. He would like to take this opportunity to thank ETH Z\"urich for its hospitality and financial support.}}
\date{}

\makeindex

\begin{document}

 \maketitle
 \begin{abstract}
  This article is a continuation of the recent paper
  \cite{Grohs2013} by the first author, where off-diagonal-decay properties (often referred to as 'localization' in the literature)
  of Moore-Penrose pseudoinverses of (bi-infinite) matrices
  are established, whenever the latter possess similar  off-diagonal-decay properties. This problem is especially interesting
  if the matrix arises as a discretization of an operator with
  respect to a frame or basis. Previous work on this problem
  has been restricted to wavelet- or Gabor frames. In \cite{Grohs2013}
  we extended these results to frames of parabolic molecules, including
  curvelets or shearlets as special cases. The present paper extends
  and unifies these results by establishing analogous properties
  for frames of $\alpha$-molecules as introduced in recent work
  \cite{Grohs2013a}. Since wavelets, curvelets, shearlets, ridgelets
  and hybrid shearlets all constitute instances of $\alpha$-molecules,
  our results establish localization properties for all these systems
  simultaneously.
 \end{abstract}
{\bf AMS Classifiers: }Primary 41AXX, Secondary 41A25, 53B, 22E.

 \section{Introduction}
This article, which is continuation of our earlier work \cite{Grohs2013}, studies off-diagonal decay properties
of Moore-Penrose pseudoinverses $A^+$ of symmetric (bi-infinite) matrices
$A=(A_{\lambda,\lambda'})_{\lambda,\lambda'\in\Lambda}$,
with $\Lambda$ a discrete index set.
More precisely, our results are of the following general type:
assume that $A$ is \emph{localized}, in the sense that
\begin{equation}\label{eq:locintro}
	|A_{\lambda,\lambda'}|\le C\omega(\lambda,\lambda')^{-N}\quad
	\mbox{for all }\lambda,\ \lambda'\in \Lambda
\end{equation}
 \new{with respect to some nice function $\omega$ measuring the distance between the indices}.
 Then the Moore-Penrose pseudoinverse of $A$ satisfies the analogous
 inequality with a different constant $C$ and a parameter 
 $N^+\le N$ which we describe explicitly.
 
 Typically $A$ arises as a 
 Gram matrix $A=\left(\langle\psi_\lambda,\psi_{\lambda'}\rangle_\mathcal{H}\right)_{\lambda,\ \lambda'\in\Lambda}$ of a
 frame $(\psi_{\lambda})_{\lambda\in \Lambda}$ of a Hilbert
 space $\mathcal{H}$. In that case the Moore-Penrose pseudoinverse  
 $A^+$ corresponds to the Gram matrix of the canonical dual frame
 of $(\psi_{\lambda})_{\lambda\in \Lambda}$. Hence, localization properties of $A^+$ provides useful information about the canonical
 dual frame. For more information regarding frames we refer to \cite{Christensen2003}. For a more detailed motivation of the problem that
 we consider in the present paper (for instance in the context of operator compression) we refer to our earlier 
 work \cite{Grohs2013}.
 
 The 'localization problem' as described above has been studied
 in several contexts, see 
 \cite{Aldroubi2008,Balan2006,Balan2006a,Balan2008,Baskakov2011,Baskakov1997,Baskakov1997a,Grochenig2004,Fornasier2005,Cordero2004,Sun2007,Sun2011,Demko1984,Jaffard1990,Futamura2009,Krishtal2011}.
In these works the index set $\Lambda$ arises as a sampling
set for either Gabor- or wavelet frames. In both cases
there exists a canonical index distance function $\omega$
for which localization results have been established in 
the aforementioned works. Recently, these results have been
extended to anisotropic frame systems such as curvelets
\cite{Candes2004a}
or shearlets \cite{Labate2005}, and more generally parabolic molecules \cite{Grohs2011}.

The present paper extends and unifies these results.
More precisely, we shall prove localization results for 
index distance functions $\omega$ which are associated
with frames of so-called $\alpha$-molecules as introduced in 
\cite{Grohs2013a}. The notion of $\alpha$-molecules
includes wavelets, ridgelets \cite{CandesPhD,GrohsRidge}, shearlets, curvelets, parabolic
molecules and $\alpha$-shearlets \cite{Keiper2013} as special cases. Consequently, the results of the present paper are applicable
to all these systems at once.

{\bf Outline. }We proceed as follows. In Section \ref{sec:framework} we provide an abstract framework for
index distance functions in which localization results can
be established. The main result of this section is 
Theorem \ref{thm_main}, which states 
that, if an index distance function $\omega$ satisfies certain properties, then localization of a matrix $A$ in 
the sense of \ref{eq:locintro} implies a similar property
for its Moore-Penrose pseudoinverse $A^+$. To further motivate
the importance of localization properties we also provide several results stating that localized matrices are
automatically bounded on a wide class of weighted $\ell^p$ Banach spaces.

Then in Section \ref{sec:alpha} we apply the abstract
framework of Section \ref{sec:framework} to specific index
distance functions, namely those associated with frames of 
$\alpha$-molecules as introduced in \cite{Grohs2013a}. 
More precisely, we verify that those index distance functions
satisfy the assumptions of the abstract theory developed
in Section \ref{sec:framework} and hence provide localization
results for the whole class of $\alpha$-molecules.

We collect some auxiliary results in Appendix \ref{appendix}.
 \section{Abstract Framework}\label{sec:framework}
 In the present section we set the abstract framework which we later
 apply in Section \ref{sec:alpha} to establish localization results
 for frames of $\alpha$-molecules.
 
 Subsection \ref{subsec:basicnotions} below starts
 by introducing the kind of index distance functions $\omega$ with which we are working.
 We consequently define the Banach space of localized matrices,
 for which a submultiplicativity property is established in Theorem \ref{thm_submult}.
 This property provides a key technical tool to prove the main result of the section,
 namely Theorem \ref{thm_main}.
 In Subsection \ref{subsec:bounded}
 we show that localization with respect to such index functions
 implies boundedness on a large range of weighted $\ell^p$ spaces.
 Finally, in Subsection \ref{subsec:invclos} we establish Theorem \ref{thm_main},
 which states the localization
 of the Moore-Penrose pseudoinverse of matrices which are
 localized with respect to $\omega$ as introduced in Subsection
 \ref{subsec:basicnotions}.
 Most of the material in this section is well-known. Using the proof techniques developed in 
 \cite{Grohs2013}, Theorem \ref{thm_main} is not too hard to establish.
 The difficult part of the present paper is
 contained in Section \ref{sec:alpha}, where we shall verify that
 canonical index distances associated to $\alpha$-molecules
 fit into the abstract framework developed in the present section.
 \subsection{Basic Notions}\label{subsec:basicnotions}
We shall prove a localization result in a general framework which we describe in the present section. Here
we introduce the notations and definitions which we shall use, starting with the following definition of an \emph{index distance} function.
  \begin{Def} \label{def:index_dist}
   Let $\Lambda$ be a discrete index set.
   An \emph{index distance} is a function $ \omega : \Lambda \times \Lambda \longrightarrow [1,\infty) $ such that
   there exist constants $ C_S,\,C_T \geq 1 $ with
   \begin{enumerate}[(i)]
    \item\label{item:pseudo_sep}
     $ \omega(\lambda,\lambda) = 1 $ for all $ \lambda \in \Lambda $;
    \item\label{item:pseudo_sym}
     $ \omega(\lambda,\lambda') \le C_S \omega(\lambda',\lambda) $ for all $ \lambda,\lambda' \in \Lambda $;
    \item\label{item:pseudo_tri}
     $ \omega(\lambda,\lambda') \le C_T \omega(\lambda,\lambda'') \omega(\lambda'',\lambda') $ for all $ \lambda,\lambda',\lambda'' \in \Lambda $.
   \end{enumerate}
  \end{Def}
    
  \begin{Def}\label{def:sep}
   We say that $\Lambda$ is \emph{separated} by $\omega$ if
   \begin{flalign*}
    & C_\Lambda := \inf_{\lambda \neq \lambda'} \omega(\lambda,\lambda') > 1 . &
   \end{flalign*}
  \end{Def}
  
  \begin{Def}\label{def:Schur}
   Let $ K \geq 1 $. We say that $\omega$ is $K$-\emph{admissible} if
   \begin{flalign*}
    & C_\omega := \sup_{\lambda \in \Lambda} \sum_{\lambda' \in \Lambda} \omega(\lambda,\lambda')^{-K} < \infty . &
   \end{flalign*}
  \end{Def}
  
  \new{
  \begin{Rmk}
   The pseudo-symmetry property \ref{def:index_dist}(\ref{item:pseudo_sym}) is not strictly necessary.
   However, our examples of index distance are all pseudo-symmetric in a natural way,
   and this allows to state the Schur type condition \ref{def:Schur} in any fixed order of the indices.
   Furthermore, one can always replace $\omega(\lambda,\lambda')$ with its symmetrization
   $ \omega^{\operatorname{sym}}(\lambda,\lambda') := \frac{1}{2} (\omega(\lambda,\lambda') + \omega(\lambda',\lambda)) $:
   if $\omega$ enjoys \ref{def:index_dist}(\ref{item:pseudo_sep}) and (\ref{item:pseudo_tri}), \ref{def:sep} and \ref{def:Schur}
   with constants $C_T$, $C_\Lambda$ and $C_\omega$,
   then $\omega^{\operatorname{sym}}$ will enjoy the same properties
   with constants $2C_T$, $ \inf_{\lambda \neq \lambda'} \omega^{\operatorname{sym}}(\lambda,\lambda') \geq C_\Lambda $
   and $ \sup_{\lambda \in \Lambda} \sum_{\lambda' \in \Lambda} \omega^{\operatorname{sym}}(\lambda,\lambda')^{-K} < 2^K C_\omega $.
   
   On the contrary, the pseudo-triangle inequality \ref{def:index_dist}(\ref{item:pseudo_tri}) is technically crucial.
  \end{Rmk}
  }
  
  Having introduced the required properties of an index distance function we now define the Banach space of \emph{localized} operators.
  \begin{Def}\label{def:loc}
   Let $\omega$ be an admissible index distance and $ N \geq 1 $.
   A matrix $ A \in \C^{\Lambda\times\Lambda} $
   is said to be $N$-\emph{localized} (with respect to $\omega$) if
   $ |A_{\lambda,\lambda'}| \lesssim \omega(\lambda,\lambda')^{-N} $ for all $ \lambda,\lambda' \in \Lambda $.
   We define $\bb_N$ as the space of all $N$-localized matrices,
   $$ \bb_N := \{ A \in \C^{\Lambda\times\Lambda} :
   |A_{\lambda,\lambda'}| \lesssim \omega(\lambda,\lambda')^{-N} \ \mbox{ for all } \lambda,\lambda' \in \Lambda \} , $$
   with associated norm
   $$
   	\|A\|_{\bb_N}:=\inf\{C>0:\,|A_{\lambda,\lambda'}| \le C \omega(\lambda,\lambda')^{-N} \mbox{ for all  }\lambda,\lambda' \in \Lambda\}
   	= \sup_{\lambda,\lambda'\in\Lambda} \omega(\lambda,\lambda')^N |A_{\lambda,\lambda'}| .
   $$
  \end{Def}
  Notice that $ \bb_N \subseteq \bb_M $ as $ N \geq M $. We next show that $\bb_N$ is complete.
  \begin{Prop}
   The set $\bb_N$ constitutes a Banach space with respect to the norm $\|\ \|_{\bb_N}$.
  \end{Prop}
  \begin{proof}
   Take a Cauchy sequence $(A_n)$ in $\bb_N$. This means that $\omega^N(A_n)$ is uniformly Cauchy.
   Moreover, $(A_n)$ is pointwise Cauchy, since
   $$ |A_n(\lambda,\lambda') - A_m(\lambda,\lambda') | \leq \| A_n - A_m \|_{\bb_N} \omega(\lambda,\lambda')^{-N} . $$
   Hence $(A_n)$ converges pointwise to some $ A \in \C^\Lambda $.
   Now $\omega^N(A_n)$ converges pointwise to $\omega^N A$, and it is uniformly Cauchy,
   therefore it converges uniformly to $\omega^N A$.
   Then, since $ \sup \omega^N A_n < \infty $ for all $n$, we also have $ \sup \omega^N A < \infty $, namely $ A \in \bb_N $.
  \end{proof}
  
  We close this subsection with the following result
  \new{regarding the action of $\bb_M$ on $\bb_N$, whenever $M$ is sufficiently large, made possible by}
  a submultiplicativity property of the $\bb_N$-norm.
  This result is crucial for the proof of our main Theorem \ref{thm_main}.
  \new{Notice that, unlike \cite[Proposition 2.13]{Grohs2013}, it is not required that $C_S=1$ or $AB$ be symmetric.}
  
  \begin{Thm}\label{thm_submult}
   Let $\Lambda$ be a discrete set, separated by a $K$-admissible index distance $\omega$.
   Let $ A \in \bb_{N+L} $ with $L\geq \max\left(2N\log_{C_\Lambda}C_T,2K\right)$, and $ B \in \bb_{N} $. Then $ AB \in \bb_N $, with
   $$
   		\|AB\|_{\bb_N} \le (1+C_\omega)\|A\|_{\bb_{N+L}}\|B\|_{\bb_N}.
   $$
  \end{Thm}
  \begin{proof}
   We have
   \begin{align*}
    |(AB)_{\lambda,\lambda'}|
    &= |\sum_{\lambda'' \in \Lambda} A_{\lambda,\lambda''} B_{\lambda'',\lambda'}| \\
    &\lesssim \sum_{\lambda'' \in \Lambda} \omega(\lambda,\lambda'')^{-N-L} \omega(\lambda'',\lambda')^{-N}
     \quad \mbox{ by \ref{def:loc}} \\
    &= \omega(\lambda,\lambda)^{-N-L} \omega(\lambda,\lambda')^{-N}
     + \sum_{\lambda'' \neq \lambda} \omega(\lambda,\lambda'')^{-N-L} \omega(\lambda'',\lambda')^{-N} \\
    &= \omega(\lambda,\lambda')^{-N} + \sum_{\lambda'' \neq \lambda} \omega(\lambda,\lambda'')^{-N-L} \omega(\lambda'',\lambda')^{-N}
     \quad \mbox{ by \ref{def:index_dist}(\ref{item:pseudo_sep})} \\
    &= \omega(\lambda,\lambda')^{-N}
     + \sum_{\lambda'' \neq \lambda} [\omega(\lambda,\lambda'') \omega(\lambda'',\lambda')]^{-N} \omega(\lambda,\lambda'')^{-L} \\
    &\leq \omega(\lambda,\lambda')^{-N} + {C_T}^N \omega(\lambda,\lambda')^{-N} \sum_{\lambda'' \neq \lambda} \omega(\lambda,\lambda'')^{-L}
     \quad \mbox{ by \ref{def:index_dist}(\ref{item:pseudo_tri})} \\
    &= \omega(\lambda,\lambda')^{-N}
     \left( 1 + {C_T}^N \sum_{\lambda'' \neq \lambda} \omega(\lambda,\lambda'')^{-L/2}\omega(\lambda,\lambda'')^{-L/2} \right) \\
    &\leq \omega(\lambda,\lambda')^{-N}
     \left( 1 + {C_T}^N C_\Lambda^{-L/2} \sum_{\lambda'' \neq \lambda} \omega(\lambda,\lambda'')^{-L/2} \right)
     \quad \mbox{ by \ref{def:sep}} \\
    &\leq \omega(\lambda,\lambda')^{-N} \left( 1 + \sum_{\lambda'' \neq \lambda} \omega(\lambda,\lambda'')^{-L/2} \right)
     \quad \mbox{ as $ L \geq 2N\log_{C_\Lambda}{C_T} $} \\
    &\leq \omega(\lambda,\lambda')^{-N} (1 + C_\omega)
     \quad \mbox{ as $ L \geq 2K $, by \ref{def:Schur}.} \qedhere
   \end{align*}
  \end{proof}
 
  \subsection{Localization implies Boundedness}\label{subsec:bounded}
  One important feature of localization is that it implies boundedness
  on a large class of weighted $\ell^p$ spaces.
A classical instance of this type of results are boundedness
results for
Calder\`on-Zygmund operators on Besov spaces, which
can be shown by representing the operators in a wavelet basis
and using localization, together with the fact that Besov
space norms can be characterized in terms of weighted $\ell^p$ norms of wavelet coefficients 
\cite{Frazier1991}.
As another example we mention Fourier integral operators
which can be shown to be localized if represented
in a frame of parabolic molecules \cite{Grohs2013,Candes2004}. Consequently, such operators
are bounded on the associated functions spaces, as described
e.g. in \cite{Borup2007}.

In the present section we establish results stating
that localized matrices always induce bounded operators
on weighted $\ell^p$ spaces, whenever the index distance
$\omega$ satisfies certain admissibility properties. 

  Given a weight function $ \w : \Lambda \to (0,\infty) $, for $ p \in (0,\infty] $ we define the weighted $\ell^p$ spaces
  $$ \ell^p_\w(\Lambda) := \{ a \in \C^\Lambda : a\w \in \ell^p(\Lambda) \} $$
  with weighted norms
  $$ \| a \|_{p,\w} := \|a\w\|_p , $$
  where we write $a\w=(a(\lambda)\w(\lambda))_{\lambda\in \Lambda}$.
  We recall the following weighted version of the Schur test (\cite[Lemma 4]{GS}).
  
  \begin{Lemma} \label{lemma:Schur1}
   Let $ A \in \C^{\Lambda\times\Lambda} $.
   For $ \w_1,\w_2 : \Lambda \to (0,+\infty) $ and $ p_0 \in (0,1] $, consider the Schur conditions
   \begin{subequations} \label{eq:weighted_Schur1}
    \begin{align}
     & \sum_{\lambda \in \Lambda} \w_2(\lambda)^{p_0} |A_{\lambda,\lambda'}|^{p_0} \leq C_1^{p_0} \w_1(\lambda')^{p_0} \mbox{ for some $ C_1 > 0 $,} \label{eq:w+Schur1} \\
     & \sum_{\lambda' \in \Lambda} |A_{\lambda,\lambda'}| \w_1(\lambda')^{-1} \leq C_2 \w_2(\lambda)^{-1} \mbox{ for some $ C_2 > 0 $,} \label{eq:w-Schur1}     
    \end{align}
   \end{subequations}
   and define the formal matrix operator
   \begin{equation*}
    (Aa)_\lambda := \sum_{\lambda'\in\Lambda} A_{\lambda,\lambda'} a_{\lambda'} \qquad a \in \C^\Lambda .
   \end{equation*}
   Then:
   \begin{enumerate}[(a)]
    \item \label{item:1-Schur1} \new{if $A$ enjoys \ref{eq:w+Schur1}, then it is bounded from $\ell^{p}_{\w_1}(\Lambda)$ to $\ell^{p}_{\w_2}(\Lambda)$
     for all $ p \in [p_0,1] $;}
    \item \label{item:inf-Schur1} if $A$ enjoys \ref{eq:w-Schur1}, then it is bounded from $\ell^\infty_{\w_1}(\Lambda)$ to $\ell^\infty_{\w_2}(\Lambda)$;
    \item \label{item:p-Schur1} if $A$ enjoys \ref{eq:w+Schur1} and \ref{eq:w-Schur1},
     then it is bounded from $\ell^p_{\w_1}(\Lambda)$ to $\ell^p_{\w_2}(\Lambda)$ for all $ p \in [p_0,\infty] $.
   \end{enumerate}
   In each case, $ \|A\|_{\ell^p_{\w_1}\to\ell^p_{\w_2}} \leq C_1^{1/p} C_2^{1/p'} $ for $ p \in [1,\infty] $ (where $ 1/p + 1/p' = 1 $ and $ 1/\infty = 0 $),
   and $ \|A\|_{\ell^p_{\w_1}\to\ell^p_{\w_2}} \leq C_1 $ for $ p \in (p_0,1] $.
  \end{Lemma}
  \begin{proof}
   \new{
   First notice that, if \ref{eq:w+Schur1} is true for $ p_0 \in (0,1] $, then it holds true with $ p_0 = 1 $ by the $p$-triangle inequality.}
   Further
   \begin{align*}
    \| Aa \|_{p_0,\w_2}^{p_0} &= \sum_{\lambda\in\Lambda} \left| \sum_{\lambda'\in\Lambda} A_{\lambda,\lambda'} a_{\lambda'} \right|^{p_0} \w_2(\lambda)^{p_0} \\
                              &\leq \sum_{\lambda\in\Lambda} \sum_{\lambda'\in\Lambda} |A_{\lambda,\lambda'}|^{p_0} |a_{\lambda'}|^{p_0} \w_2(\lambda)^{p_0} \\
                              &= \sum_{\lambda'\in\Lambda} |a_{\lambda'}|^{p_0} \sum_{\lambda\in\Lambda} \w_2(\lambda)^{p_0} |A_{\lambda,\lambda'}|^{p_0} \\
                              &\leq C_1^{p_0} \sum_{\lambda'\in\Lambda} |a_{\lambda'}|^{p_0} \w_1(\lambda')^{p_0} \\
                              &= C_1^{p_0} \| a \|_{p_0,\w_1}^{p_0} .
   \end{align*}
   \new{Therefore, the interpolation theorem (\cite[Corollary 2.2]{Gustavsson1982}) yields item (\ref{item:1-Schur1}).}
   Now, assuming \ref{eq:w-Schur1} we can estimate
   \begin{align*}
    \| Aa \|_{\infty,\w_2} &\leq \sup_{\lambda\in\Lambda} \sum_{\lambda'\in\Lambda} |A_{\lambda,\lambda'}| |a_{\lambda'}| \w_2(\lambda) \\
                           &= \sup_{\lambda\in\Lambda} \w_2(\lambda) \sum_{\lambda'\in\Lambda} |A_{\lambda,\lambda'}| \w_1(\lambda')^{-1} |a_{\lambda'}| \w_1(\lambda') \\
                           &\leq \sup_{\lambda\in\Lambda} \w_2(\lambda) \sum_{\lambda'\in\Lambda} |A_{\lambda,\lambda'}| \w_1(\lambda')^{-1} \sup_{\lambda'\in\Lambda} |a_{\lambda'}| \w_1(\lambda') \\
                           &\leq C_2 \sup_{\lambda\in\Lambda} \w_2(\lambda) \w_2(\lambda)^{-1} \sup_{\lambda'\in\Lambda} |a_{\lambda'}| \w_1(\lambda') \\
                           &= C_2 \| a \|_{\infty,\w_1} ,
   \end{align*}
   whence we obtain item (\ref{item:inf-Schur1}).
   Finally assume both \ref{eq:w+Schur1} and \ref{eq:w-Schur1}.
   Then, by the interpolation theorem (\cite[Corollary 2.2]{Gustavsson1982}),
   items (\ref{item:1-Schur1}) and (\ref{item:inf-Schur1}) imply item (\ref{item:p-Schur1}).
   
   The estimate for $ \|A\|_{\ell^p_{\w_1}\to\ell^p_{\w_2}} $ with $ p \in (p_0,1] $ follows from \cite[Proposition 1.1, Theorem 2.4]{Gustavsson1982},
   interpolating between $p_0$ and $1$.
   As for the case $ p \in [1,\infty] $, the bound follows easily by applying the H\"older inequality, \ref{eq:w-Schur1} and \ref{eq:w+Schur1} with $p_0 = 1$
   (see \cite[Lemma 4]{GS}).
   \end{proof}
  
  If a matrix $ A \in \C^{\Lambda\times\Lambda} $ decays with respect to some bounding function (e.g an index distance),
  $$ |A_{\lambda,\lambda'}| \lesssim \omega(\lambda,\lambda')^{-N}, $$
  one can test the boundedness of $A$ by testing estimates of the form
  $$ \sum_{\lambda \in \Lambda} \w_2^{p_0}(\lambda) \omega(\lambda,\lambda')^{-p_0K} \lesssim \w_1^{p_0}(\lambda') , \quad
  \sum_{\lambda' \in \Lambda} \omega(\lambda,\lambda')^{-K} \w_1^{-1}(\lambda') \lesssim \w_2^{-1}(\lambda), $$
  which imply conditions \ref{eq:weighted_Schur1} for $ N $ sufficiently large.
  In oder to give a precise statement,
  we introduce the concept of \emph{admissibility} with respect to two weight sequences $\w_1,\w_2$ and a root $p_0$.
  \begin{Def}\label{def:pKadmiss1}
  Let $ \w_1,\w_2 :\Lambda \to (0,+\infty) $, $ p_0 \in (0,1] $ and $ K \geq 1 $.
  An index distance $\omega$ is called $(\w_1,\w_2,p_0,K)$-\emph{admissible} if
  $$ \sum_{\lambda \in \Lambda} \w_2^{p_0}(\lambda) \omega(\lambda,\lambda')^{-p_0K} \leq C_1^{p_0} \w_1^{p_0}(\lambda') , \quad
  \sum_{\lambda' \in \Lambda} \omega(\lambda,\lambda')^{-K} \w_1^{-1}(\lambda') \leq C_2 \w_2^{-1}(\lambda), $$
  for some $C_1,\ C_2>0$.
  \end{Def}
  Note that, thanks to property \ref{def:index_dist}(\ref{item:pseudo_sym}),
  $K$-admissibility as defined in \ref{def:Schur} is equivalent to
  $(1,1,1,K)$-admissibility as defined in \ref{def:pKadmiss1}.
  \begin{Prop}
   Let $\omega$ be a $(\w_1,\w_2,p_0,K)$-admissible index distance.
   If $ A \in \bb_N $ for some $ N \geq K $, then it defines a bounded operator
   from $\ell^p_{\w_1}(\Lambda)$ to $\ell^p_{\w_2}(\Lambda)$ for all $ p \in [p_0,\infty]$, with
   \begin{align*}
    & \|A\|_{\ell^p_{\w_1}\to\ell^p_{\w_2}} \leq C_1^{1/p}C_2^{1/p'}\|A\|_{\bb_N} \quad p \in [1,\infty] , \\
    & \|A\|_{\ell^p_{\w_1}\to\ell^p_{\w_2}} \leq C_1\|A\|_{\bb_N} \quad p \in (p_0,1].
   \end{align*}
   If $\omega$ is $K$-admissible, then $A$ defines a bounded operator
   from $\ell^p(\Lambda)$ to $\ell^p(\Lambda)$ for all $ p \in [1,\infty]$, with
   $$ \|A\|_{\ell^p\to\ell^p} \leq C_S^{N/p} C_\omega \|A\|_{\bb_N} \quad p \in [1,\infty] . $$
  \end{Prop}
  \begin{proof}
   The proof follows directly by applying Lemma \ref{lemma:Schur1}.
  \end{proof}
  
  \subsection{Inverse Closedness}\label{subsec:invclos}

 For several applications it is important to know the localization 
 properties of the operator $A^{-1}$, assuming that $A$, restricted to its image, constitutes
 an isomorphism $A:\ell^2(\Lambda)\to \ell^2(\Lambda)$. For
 instance, as we have seen in the previous subsection, if it can be shown that $A^{-1}\in \bb_N$ for
 sufficiently large $N$ one can deduce the boundedness of $A^{-1}$
 on a large class of sequence spaces.
 Except for very special cases of $\omega$ it cannot be expected
 that $A^{-1}\in \bb_N$ if $A\in \bb_N$. However, we shall show
 that the Moore-Penrose pseudoinverse $A^+ \in \bb_{N^+}$ whenever $A\in \bb_N$, where
 $N^+\le N$, depending only on $\omega$ and the spectrum of $A$.
 Moreover, this dependence will be made completely explicit.
  
  We now describe the spectral assumption on $A$, which we
  shall impose in our analysis.
\begin{Def}
	The matrix $A$ viewed as an operator
	from $\ell^2(\Lambda)$ to itself possesses 
	a \emph{spectral gap}
	if there exist numbers $0<a\le b<\infty$
	such that
	\begin{equation*}
		\sigma_2\left(A\right)
		\subset \left\{0\right\}\cup [a,b],
	\end{equation*}
	where $\sigma_2\left(A\right)$ denotes the $\ell^2(\Lambda)$-spectrum of $A$.
\end{Def}
It $A$ is symmetric and possesses a spectral gap we 
can define its Moore-Penrose pseudoinverse
$A^+$
which satisfies the normal equations
\begin{equation}\label{eq:MP}
	A^2A^+ = A.
\end{equation}
Having stated all necessary definitions we can now state
our main result.
\begin{Thm}\label{thm_main}
Assume that $A\in \bb_{N+L}$ with 
\begin{equation} \label{eq:hp}
	N \geq K \qquad L\geq\max \left(2N\log_{C_\Lambda}C_T,2K\right)
\end{equation}
is symmetric and possesses a spectral gap, e.g.,
$$
	\sigma_2(A) \subset \{0\}\cup [a,b].
$$
Then with $A^+$ denoting
its Moore-Penrose pseudoinverse
we have 
$$
	A^+\in \bb_{N^+}
$$
with 
\begin{equation}\label{eq:nplus}
	N^+ = N
	\left(1-\frac{\log\left(1 + 
	\frac{2}{a^2+b^2} \|A\|_{\bb_{N+L}}^2\left(1+C_\omega\right)^2\right)}	
		{\log\left(\frac{b^2 - a^2}{b^2 + a^2}\right)}\right)^{-1}.
\end{equation}

\end{Thm}

 \begin{proof}
 	The proof goes exactly as the proof
 	of  \cite[Theorem 2.12]{Grohs2013}, using 
 	our submultiplicativity result, Theorem \ref{thm_submult}.
 \end{proof}
 \section{Application to $\alpha$-Molecules}\label{sec:alpha}
 We intend to apply the general results of the previous section to
 the study of $\alpha$-molecules \cite{Grohs2013a}. This class of systems includes
 as special cases wavelets, curvelets, shearlets, hybrid shearlets and
 ridgelets, therefore our results will allow us to gain localization results for all these
 systems simultaneously. 
 We proceed as follows.
 In Subsection \ref{subsec:indexdist} we describe the index distance
 $\omega$
 which has been introduced in \cite{Grohs2013a} and which is defined
 on a contiuous phase space $P$. Then we prove that this function
 $\omega$ satisfies all the assumptions of Definition \ref{def:index_dist}. This turns out to be the most technical part
 of this work. Then, only later in Subsection \ref{subsec:alphamol}
 we briefly introduce the notion of $\alpha$-molecules. In a system
 of $\alpha$-molecules, every function is associated with a point in
 the phase space $P$ and therefore every such system is associated
 to a discrete sampling set $\Lambda\subset P$. 
 We discuss two canonical choices of $\Lambda$ in detail: 
 so-called curvelet-type systems in Subsubsection \ref{subsubsec:curve} and so-called shearlet-type systems
 in Subsubsection \ref{subsubsec:shear}. In both cases
 we show that the index distance $\omega$ restricted to 
 $\Lambda$ is separated and admissible. In summary 
 $\alpha$-molecules, together with the index distance introduced
 in \cite{Grohs2013a}, fits into the abstract framework developed
 earlier in Section \ref{sec:framework}.
 As an application we present localization for canonical duals
 of frames of $\alpha$-curvelets (Theorem \ref{thm:curveloc}) and $\alpha$-shearlets (Theorem \ref{thm:shearloc}).
 \subsection{Index Distance}\label{subsec:indexdist}
 Before we describe the notion of $\alpha$-molecules
 we start by defining the corresponding index distance
 $\omega_\alpha$
 and show that our main result can indeed be applied to this
 index distance.
 Roughly speaking, $\alpha$-molecules can be associated
 with a scale, an orientation and a location. 
  Therefore we first define a contiuous parameter space $P$ as the product
  \begin{equation}
   P := \R_+ \times S^1 \times \R^2 .
  \end{equation}
  The parameters in $P$ will be denoted by $ p = (s,\theta,x) $, $ p' = (s',\theta',x') $, and so on.
  We also define, for each $ \alpha \in [0,1] $, the function
  \begin{equation} \label{eq:alpha-index_dist}
   \omega_\alpha := M(1+d_\alpha) ,
  \end{equation}
  where
  \begin{align*}
   & M(p,p') := \max(s/s',s'/s) \\
   & d_\alpha(p,p') := \min(s,s')^{2(1-\alpha)}|\theta-\theta'|^2 + \min(s,s')^{2\alpha}\|x-x'\|^2 + \min(s,s') |\langle x-x' , e_\theta \rangle| ,
  \end{align*}
  $e_\theta$ being the ``co-direction'' $(\cos\theta,-\sin\theta)$.
  We shall often adopt the abbreviations $ \Delta\theta := \theta - \theta' $ and $ \Delta x := x - x' $.
  
  \new{
  \begin{Rmk}
   This definition of $\omega_\alpha$ differs from the one presented in \cite{Grohs2013a},
   where the last term of $d_\alpha(p,p')$ is replaced by
   $$ \frac{\min(s,s')^2 |\langle \Delta x , e_\theta  \rangle|^2}{1 + \min(s,s')^{2(1-\alpha)}|\Delta\theta|^2} . $$
   However, an application of the inequality of arithmetic and geometric means yields
   \begin{align*}
         \ & 1 + \min(s,s')^{2(1-\alpha)}|\Delta\theta|^2 + \frac{\min(s,s')^2 |\langle \Delta x , e_\theta  \rangle|^2}{1 + \min(s,s')^{2(1-\alpha)}|\Delta\theta|^2} \\
    =    \ & \left(\sqrt{1 + \min(s,s')^{2(1-\alpha)}|\Delta\theta|^2}\right)^2 + \left( \frac{\min(s,s') |\langle \Delta x , e_\theta  \rangle|}{\sqrt{1 + \min(s,s')^{2(1-\alpha)}|\Delta\theta|^2}} \right)^2 \\
    \geq \ & 2 \sqrt{1 + \min(s,s')^{2(1-\alpha)}|\Delta\theta|^2} \frac{\min(s,s') |\langle \Delta x , e_\theta  \rangle|}{\sqrt{1 + \min(s,s')^{2(1-\alpha)}|\Delta\theta|^2}} \\
    =    \ & 2 \min(s,s') |\langle \Delta x , e_\theta  \rangle| ;
   \end{align*}
   then, if we call $\tilde{\omega}_\alpha$ the index distance introduced in \cite{Grohs2013a}, we have $ \tilde{\omega}_\alpha \geq 2 \omega_\alpha $.
   Now, the key concept to preserve here is \emph{almost orthogonality} (see \cite{Grohs2013a} for details),
   and this inequality shows exactly that systems which are almost orthogonal respect to $\tilde{\omega}_\alpha$
   are still almost orthogonal respect to $\omega_\alpha$.
   
   Also note that, for suitable choices of the parameters,
   $\omega_1$ corresponds to the wavelet index distance,
   whereas $\omega_{\frac{1}{2}}$ returns the curvelet-shearlet index distance studied in \cite{Grohs2013}.
  \end{Rmk}
  }
  
  The restriction of $\omega_\alpha$ to any discrete index set $ \Lambda \subset P $ describes
  an index distance as we show in the next result.
  \begin{Prop}
   $\omega_\alpha$ is an index distance for all $ \alpha \in [0,1] $
   and all discrete index sets $\Lambda \subset P$.
   The resulting constants obey $ C_S \leq 2 $ and $ C_T \leq 4 $.
  \end{Prop}
  \begin{proof}
   For ease of notation, we shall avoid to specify the index $\alpha$ in $\omega_\alpha$.
   It is apparent that $\omega$ is $1$ on the diagonal, that is property \ref{def:index_dist}(\ref{item:pseudo_sep}).
   We next prove properties \ref{def:index_dist}(\ref{item:pseudo_sym}) and (\ref{item:pseudo_tri}).
   
   \ref{def:index_dist}(\ref{item:pseudo_sym}).
   Notice that the only non symmetric term in $\omega(p,p')$ is $ \min(s,s') |\langle \Delta x , e_\theta \rangle| $,
   so it sufficies to show that $|\langle \Delta x , e_\theta \rangle| \lesssim \min(s,s')^{-1} d(p',p) $.
   Since
   $$ |\langle \Delta x , e_\theta \rangle| \leq |\langle \Delta x , e_\theta \rangle| + |\langle \Delta x , e_{\theta'} \rangle|
   \leq |\langle \Delta x , e_\theta \rangle - \langle \Delta x , e_{\theta'} \rangle| + 2 |\langle \Delta x , e_{\theta'} \rangle| , $$
   it remains to estimate
   $$ |\langle \Delta x , e_\theta \rangle - \langle \Delta x , e_{\theta'} \rangle| = |\langle \Delta x , e_\theta - e_{\theta'} \rangle|
   \leq \|e_\theta - e_{\theta'}\| \|\Delta x\| . $$
   By prosthaphaeresis
   $$ e_\theta - e_{\theta'} = (\cos\theta - \cos\theta',\sin\theta - \sin\theta')
   = 2 \sin\left(\frac{\theta-\theta'}{2}\right) (-\sin\left(\frac{\theta+\theta'}{2}\right) , \cos\left(\frac{\theta+\theta'}{2}\right)) , $$
   then
   $$ \|e_\theta - e_{\theta'}\| = 2 \left|\sin\left(\frac{\theta-\theta'}{2}\right)\right| \leq 2 \left|\frac{\theta-\theta'}{2}\right| = |\theta-\theta'| . $$
   Therefore
   \begin{align*}
    |\langle \Delta x , e_\theta \rangle - \langle \Delta x , e_{\theta'} \rangle|
     &\leq |\Delta\theta| \|\Delta x\| \\
     &= \min(s,s')^{1/2-\alpha} |\Delta\theta| \min(s,s')^{\alpha-1/2} \|\Delta x\| \\
     &\leq \frac{1}{2} (\min(s,s')^{1-2\alpha} |\Delta\theta|^2 + \min(s,s')^{2\alpha-1} \|\Delta x\|^2) \\
     &\leq \frac{1}{2} \min(s,s')^{-1} d(p,p')
   \end{align*}
   by the inequality of arithmetic and geometric means.
   Thus we get \ref{def:index_dist}(\ref{item:pseudo_sym}) with bound $ C_S \leq 2 $.
   
   \ref{def:index_dist}(\ref{item:pseudo_tri}).
   We shall write indices $01$ for the expressions evaluated in $(p,p')$, $02$ for $(p,p'')$ and $21$ for $(p'',p')$.
   The letter $s_{01}$ will be a short for $\min(s,s')$, and similarly with regard to $s_{02}$ and $s_{21}$.
   
   Well then, we have to show that $ \omega_{01} \leq C \omega_{02}\omega_{21} $ for some constant $ C \geq 1 $. Write
   \begin{align*}
    \omega_{01} &= M_{01}(1 + d_{01}) \\
                &= M_{01}(1 + s_{01}^{2-2\alpha}\Delta\theta^2 + s_{01}^{2\alpha}\Delta x^2 + s_{01} |\langle \Delta x , e_\theta \rangle| ) \\
                &= \underbrace{M_{01}(s_{01}^{2-2\alpha}\Delta\theta^2 + s_{01}^{2\alpha}\Delta x^2)}_{=:A}
                   + \underbrace{M_{01} (1 + s_{01}|\langle \Delta x , e_\theta \rangle|)}_{=:B} .
   \end{align*}
   We shall prove that
   \begin{align}
    &A \leq 2 M_{02}M_{21}(d_{02} + d_{21}) \leq 2 \omega_{02}\omega_{21} , \label{eq:A} \\
    &B \leq 2 M_{02}M_{21}(1 + d_{02} + d_{21} + d_{02}d_{21}) = 2 \omega_{02}\omega_{21} , \label{eq:B}
   \end{align}
   so that $ \omega_{01} = A + B \leq 4 \omega_{02}\omega_{21} $, as we claim.
   
   In order to verify \ref{eq:A} and \ref{eq:B}, we first observe that $\omega$ is translational invariant, namely
   $$ \omega((s,\theta,x),(s',\theta',x')) = \omega ((s,\theta,x+t),(s',\theta',x'+t)) \quad \forall t \in \R^2 , $$
   and therefore we can set $ x = 0 $.
   Moreover, we can work in coordinates $ e_\theta , e_\theta^\perp $, so that $ e_\theta = (1,0) $ and $ \theta = 0 $.
   Coordinates for $x'$ and $x''$ are called $(x_1,y_1)$ and $(x_2,y_2)$, respectively.
   With this choices we have:
   \begin{align*}
    &d_{01} = s_{01}^{2-2\alpha} |\theta'|^2 + s_{01}^{2\alpha} (|x_1|^2 + |y_1|^2) + s_{01}|x_1| , \\
    &d_{02} = s_{02}^{2-2\alpha} |\theta''|^2 + s_{02}^{2\alpha} (|x_2|^2 + |y_2|^2) + s_{02}|x_2| , \\
    &d_{21} = s_{21}^{2-2\alpha} |\theta''-\theta'|^2 + s_{21}^{2\alpha} (|x_2-x_1|^2 + |y_2-y_1|^2)
              + s_{21}|\cos\theta''(x_2-x_1) + \sin\theta''(y_2-y_1)| .
   \end{align*}
   
   Let us start with \ref{eq:A}. We estimate
   \begin{align*}
    M_{01}s_{01}^{2-2\alpha}|\theta'|^2 &\leq 2 (M_{01}s_{01}^{2-2\alpha}|\theta''|^2 + M_{01}s_{01}^{2-2\alpha}|\theta''-\theta'|^2 ) \\
                                        &\leq 2 (M_{02}M_{21}s_{02}^{2-2\alpha}|\theta''|^2 + M_{02}M_{21}s_{21}^{2-2\alpha}|\theta''-\theta'|^2)
   \end{align*}
   by Lemma \ref{lemma:trimin} (just exponentiate the appropriate inequality to use it in multiplicative form). Likewise we have
   \begin{align*}
    M_{01}s_{01}^{2\alpha}\|x'\|^2 &\leq 2 (M_{01}s_{01}^{2\alpha}\|x''\|^2 + M_{01}s_{01}^{2\alpha}\|x''-x'\|^2) \\
                                   &\leq 2 (M_{02}M_{21}s_{02}^{2\alpha}\|x''\|^2 + M_{02}M_{21}s_{21}^{2\alpha}\|x''-x'\|^2) ,
   \end{align*}
   again by \ref{lemma:trimin} (in multiplicative form). Thus we get \ref{eq:A}.
   
   Now we move on to \ref{eq:B}. One has
   $$ B = M_{01} + M_{01}s_{01}|x_1| \leq M_{02}M_{21} + M_{01}s_{01}|x_1| ; $$
   then, if we show that
   \begin{equation*}
    \underbrace{M_{01}s_{01}|x_1|}_{=: L} \leq \underbrace{M_{02}M_{21}[1 + 2(d_{02} + d_{21} + d_{02}d_{21})]}_{=: R} ,
   \end{equation*}
   we have \ref{eq:B}. First notice that
   \begin{align*}
   L &= \frac{\max(s,s')}{\min(s,s')} \min(s,s') |x_1| \\
     &= \max(s,s')|x_1| .
   \end{align*}
   Our estimates for $R$ depend on how large is the angle $\theta''$.
   We shall occasionally write $s_{012}$ for $\min(s,s',s'')$.
      
   \begin{description}
    \item[$ |\theta''| \geq \pi/4 $.]
     We have
     \begin{align*}
      R &\geq M_{02}M_{21}[1 + 2(s_{02}|x_2| + s_{02}^{2-2\alpha}|\theta''|^2 s_{21}^{2\alpha}|x_2-x_1|^2)] \\
        &\geq M_{02}M_{21}[1 + 2(s_{02}|x_2| + \frac{\pi^2}{16}s_{02}^{2-2\alpha}s_{21}^{2\alpha}|x_2-x_1|^2)] \\
        &\geq M_{02}M_{21}[1 + 2(s_{012}|x_2| + \frac{\pi^2}{16}s_{012}^2|x_2-x_1|^2)] =: R' .
     \end{align*}
     Dividing by $M_{02}M_{21}$, we have that $ R \geq L $ if
     $$ 1 + 2s_{012}|x_2| + \frac{\pi^2}{8}s_{012}^2|x_2-x_1|^2 \geq \frac{\max(s,s')}{M_{02}M_{21}}|x_1| . $$
     But
     \begin{align*}
      \frac{\max(s,s')}{M_{02}M_{21}} &= \max(s,s') \ \frac{\min(s,s'')}{\max(s,s'')} \ \frac{\min(s'',s')}{\max(s'',s')} \\
                                      &= \max(s,s') \ \frac{\max(\min(s,s''),\min(s'',s')) \ \min(s,s',s'')}{\min(\max(s,s''),\max(s'',s')) \ \max(s,s',s'')} \\
                                      &\leq \max(s,s') \ \frac{\min(s,s',s'')}{\max(s,s',s'')} \quad \mbox{ by Lemma \ref{lemma:maxmin}} \\
                                      &\leq \max(s,s',s'') \ \frac{\min(s,s',s'')}{\max(s,s',s'')} \\
                                      &= \min(s,s',s'') = s_{012} ,
     \end{align*}
     whence $ R \geq L $ provided that
     $$ 1 + 2a|x_2| + \frac{\pi^2}{8}a^2|x_2-x_1|^2 \geq a|x_1| $$
     for every $ a > 0 $ and $ x_1,x_2 \in \R $.
     It is actually equivalent to show this for $ a = 1 $, since we can always replace $(x_1,x_2)$ with $(ax_1,ax_2)$.
     By the usual triangle inequality, we have
     $$ |x_1| \leq |x_2-x_1| + |x_2| . $$
     Now, if $ |x_2-x_1| \geq 1 $, then $ |x_2-x_1| \leq |x_2-x_1|^2 $ and we are done.
     Otherwise $ |x_2-x_1| \leq 1 $, and we are done as well.
     
    \item[$ |\theta''| \leq \pi/4 $.]
     We have
     \begin{align*}
      R &\geq M_{02}M_{21} \{1 + 2[s_{02}|x_2| + s_{21}|\cos\theta''(x_2-x_1) - \sin\theta''(y_2-y_1)|
                    + s_{02}^{2-2\alpha}|\theta''|^2 s_{21}^{2\alpha}|y_2-y_1|^2]\} \\
        &\geq M_{02}M_{21} \{1 + 2[s_{012}|x_2| + s_{012}|\cos\theta''(x_2-x_1) - \sin\theta''(y_2-y_1)|
                    + s_{012}^2|\theta''|^2|y_2-y_1|^2]\} \\
        &\geq M_{02}M_{21} \{1 + 2[s_{012}|x_2| + s_{012}(\cos\theta''|x_2-x_1| - \sin|\theta''||y_2-y_1|)
                    + s_{012}^2|\theta''|^2|y_2-y_1|^2)]\} \\
        &\geq M_{02}M_{21} \{1 + 2s_{012}|x_2| + \sqrt{2}s_{012}|x_2-x_1| - 2s_{012}\sin|\theta''||y_2-y_1|
                    + 2s_{012}^2|\theta''|^2|y_2-y_1|^2\} \\
        &= M_{02}M_{21}\underbrace{ ( 1 - 2s_{012}\sin|\theta''||y_2-y_1| + 2s_{012}^2|\theta''|^2|y_2-y_1|^2 ) }_{=: R_1} \\
        &+ M_{02}M_{21} ( 2 s_{012}|x_2| + \sqrt{2}s_{012}|x_2-x_1| ) \\
        &\geq M_{02}M_{21}R_1 + \underbrace{ M_{02}M_{21}s_{012}(|x_2| + |x_2-x_1|) }_{=: R_2} .
     \end{align*}
     
     In $R_1$ we can regard $s_{012}$ as any $ a > 0 $, and we can actually set $ a = 1 $ by replacing $(y_1,y_2)$ with $(ay_1,ay_2)$.
     Thus we get a polynomial in $|y_2-y_1|$ with discriminant
     $$ \Delta/4 = \sin^2|\theta''| - 2 |\theta''|^2 \leq 0 , $$
     whence $ R_1 \geq 0 $. It follows that $ R \geq R_2 . $
     
     On the other hand, we have
     \begin{align*}
      M_{02}M_{21}s_{012} &= \frac{\max(s,s'')}{\min(s,s'')} \ \frac{\max(s'',s')}{\min(s'',s')} \ \min(s,s',s'') \\
                          &= \frac{\max(s,s',s'') \ \min(\max(s,s''),\max(s'',s'))}{\min(s,s',s'') \ \max(\min(s,s''),\min(s'',s'))} \ \min(s,s',s'') \\
                          &= \max(s,s',s'') \ \frac{\min(\max(s,s''),\max(s'',s'))}{\max(\min(s,s''),\min(s'',s'))} \\
                          &\geq \max(s,s',s'') \quad \mbox{ by Lemma \ref{lemma:maxmin}} \\
                          &\geq \max(s,s') ,
     \end{align*}
     whence
     \begin{align*}
      R_2 &\geq \max(s,s')(|x_2| + |x_2-x_1|) \\
          &\geq \max(s,s')|x_1| = L .
     \end{align*}
   \end{description}
   The property \ref{def:index_dist}(\ref{item:pseudo_tri}) is finally proven, with constant bound $ C_T \leq 4 $.
  \end{proof}
  
  \new{
  \begin{Rmk}
   Following \cite{Grohs2013}, one may think to write $ \omega_\alpha(p,p') = \max(s,s')(1 + \min(s,s')\tilde{d}_\alpha(p,p')) $, with
   $$ \tilde{d}_\alpha(p,p') := \min(s,s')^{1-2\alpha} |\Delta\theta|^2 + \min(s,s')^{2\alpha-1} \|\Delta x\|^2 + |\langle \Delta x , e_\theta \rangle| , $$
   and check the assumptions made in \cite{Grohs2013}.
   It turns out that all the hypothesis are satisfied, except for the very important pseudo-triangle inequality
   $$ \tilde{d}_\alpha(p,p') \lesssim \tilde{d}_\alpha(p,p'') + \tilde{d}_\alpha(p'',p') , $$
   which is true if and only if $ \alpha = \frac{1}{2} $.
   To see this, begin by fixing $ \alpha \in [0,\frac{1}{2}) $, so that $ 1-2\alpha \in (0,1] $. For any $ C \geq 1 $ pick
   $$ s = s' > C^{1/(1-2\alpha)} , \quad s''= 1 , \quad \theta = 0 , \quad \theta' = \theta'' \neq 0 , \quad x = x' = x'' = 0 , $$
   whence we obtain
   $$ \tilde{d}_\alpha(p,p') = s^{1-2\alpha} \theta'^2 > C \theta'^2 = C \tilde{d}_\alpha(p,p'') = C (\tilde{d}_\alpha(p,p'') + \tilde{d}_\alpha(p'',p')) . $$
   Similarly, if $ \alpha \in (\frac{1}{2},1] $, $ 2\alpha-1 \in (0,1] $, for any $ C \geq 1 $ we can set
   $$ s = s' > C^{1/(2\alpha-1)} , \quad s'' = 1 , \quad \theta = \theta' = \theta'' = 0 , \quad x = 0 , \quad x' = x'' = (0,y) \neq 0 , $$
   whence
   $$ \tilde{d}_\alpha(p,p') = s^{2\alpha-1} y^2 > C y^2 = C \tilde{d}_\alpha(p,p'') = C (\tilde{d}_\alpha(p,p'') + \tilde{d}_\alpha(p'',p')) . $$
  \end{Rmk}
  }
  
 \subsection{$\alpha$-Molecules}\label{subsec:alphamol}
  We now introduce the notion of $\alpha$-\emph{molecules} as
  in \cite{Grohs2013a}.
  There, $\alpha$-molecules are defined as systems of functions $(m_{\lambda})_{\lambda\in \Lambda}$,
where each $m_{\lambda}\in L^2(\mathbb{R}^2)$ has to satisfy some additional properties.
In particular, each function $m_{\lambda}$ will be associated with a unique point in $P$, which
is done via a \emph{parametrization} as defined below.
\begin{Def}A \emph{parametrization} consists of a
    pair $(\Lambda,\Phi_\Lambda)$
    where $\Lambda$ is an index set and $\Phi_\Lambda$ is a mapping
    $$
        \Phi_\Lambda:\left\{\begin{array}{ccc}\Lambda &\to & P\\
        \lambda & \mapsto & \left(s_\lambda , \theta_\lambda , x_\lambda\right)
        \end{array}\right.
    $$
    which associates with each $\lambda\in \Lambda$ a \emph{scale} $s_\lambda$,
    a \emph{direction} $\theta_\lambda$ and a \emph{location} $x_\lambda$.
\end{Def}
\new{With slight abuse of notation,
below we shall confuse $\Lambda$ with the image $\Phi_\Lambda(\Lambda)$, and $\omega$ with the pull-back $\omega\circ\Phi_\Lambda$.}

Let
\begin{equation}
D_s := \begin{pmatrix}
        s & 0 \\
        0 & s^{\alpha}
       \end{pmatrix} , \qquad R_\theta := \begin{pmatrix}
                                           \cos\theta & -\sin\theta \\
                                           \sin\theta & \cos\theta
                                          \end{pmatrix}
\end{equation}
denote respectively the anisotropic dilation matrix associated with $s>0$ and $\alpha\in[0,1]$
and the rotation matrix by an angle $\theta\in S^1$.
Now we have collected all the necessary ingredients for defining $\alpha$-molecules.

\begin{Def}
 Let $(\Lambda,\Phi_\Lambda)$ be a parametrization \new{and $R,M,N_1,N_2 >0$}.
 A family $(m_\lambda)_{\lambda \in \Lambda}$ \new{$\subset L^2(\R^2)$} is called a family of $\alpha$-\emph{molecules} with respect to $(\Lambda,\Phi_\Lambda)$ of
    order $(R,M,N_1,N_2)$, if
    it can be written as
    $$
        m_\lambda (x) =
        s_\lambda^{(1+\alpha)/2}
        a^{(\lambda)}
        \left(D_{s_\lambda}R_{\theta_\lambda}\left(x - x_\lambda\right)\right)
    $$
    such that
    \begin{equation*}
        \left| \partial^\beta \hat a^{(\lambda)}(\xi)\right|
        \lesssim \min\left(1,s_\lambda^{-1} + |\xi_1| + s_\lambda^{-(1-\alpha)}|\xi_2|\right)^M
        \left\langle |\xi|\right\rangle^{-N_1} \langle \xi_2 \rangle^{-N_2}
        \quad \mbox{for all $|\beta|\le R$}.
    \end{equation*}
    The implicit constants are uniform over $\lambda\in \Lambda$.
\end{Def}
It is instructive to look at some special cases. 
For instance the case $\alpha = 1$ corresponds to wavelet-type systems, whereas the case $\alpha = \frac{1}{2}$ corresponds
to parabolic molecules \cite{Grohs2011}, which include
curvelets \cite{Candes2004} and shearlets \cite{Labate2005}.
The case $\alpha = 0$ corresponds to ridgelet-type systems
\cite{Grohs2011a,CandesPhD}. The systems with $\alpha\in (0,\frac{1}{2})$
have been called 'hybrid' systems. Such systems, together with
their approximation properties, have been studied recently in
\cite{Keiper2013}. 

Systems of $\alpha$-molecules are useful for the decomposition and reconstruction of functions $f\in L^2(\R^2)$ in a numerically stable fashion.
To this end it is required that a system $(m_\lambda)_{\lambda\in \Lambda}$ constitues a 
\emph{frame} in the sense that there exist constants
$0<a\le b <\infty$ such that
\begin{equation}\label{eq:framedef}
	a^2\|f\|_{L^2(\R^2)}^2\le 
	\sum_{\lambda\in \Lambda}|\langle f, m_\lambda \rangle_{L^2(\R^2)}|^2
	\le b^2\|f\|_{L^2(\R^2)}^2\quad
	\mbox{for all }f\in L^2(\R^2).
\end{equation}
If \ref{eq:framedef} holds true, there exists a canonical
dual frame $(\tilde{m}_\lambda)_{\lambda\in \Lambda}$
satisfying
$$
	f = \sum_{\lambda\in \Lambda}\langle f, m_\lambda \rangle_{L^2(\R^2)}\tilde{m}_\lambda
	  = \sum_{\lambda\in \Lambda}\langle f, \tilde{m}_\lambda \rangle_{L^2(\R^2)}m_\lambda \quad
	\mbox{for all }f\in L^2(\R^2),
$$
see e.g. \cite{Christensen2003}.
\new{Unless $ a = b $, in which case $ \tilde{m}_\lambda = a^{-2} m_\lambda $,
the canonical dual frame is not in general explicitly known. Nevertheless,}
for a number of applications it is important to study
its
structure,
in particular
its similarity or dissimilarity to the primal frame
$(m_\lambda)_{\lambda\in \Lambda}$.
Again we refer
to \cite{Grohs2013} for more detailed information.
Crucial in this respect is the \emph{localization} \new{property which we define next.
\begin{Def}
 We say that a system $(m_\lambda)_{\lambda\in\Lambda}$ is $N$-\emph{localized}
 (with respect to the index distance \ref{eq:alpha-index_dist})
 if such is its Gramian, that is
 $$ \left(\langle m_\lambda, m_{\lambda'} \rangle_{L^2(\R^2)}\right)_{\lambda,\lambda'\in \Lambda} \in \bb_N .$$
\end{Def}
Notice that $\omega_\alpha$ provides a measure of the off-diagonal decay of the Gramian.
In the following we shall study conditions under which the dual of a frame of $\alpha$-molecules is localized, provided that such is the primal frame.

In order to apply the machinery of Section \ref{sec:framework} we first need to observe a couple of facts.
\begin{Lemma} \label{lemma:gram}
 Given a frame $ (m_\lambda)_{\lambda\in\Lambda} \subset L^2(\R^2) $ with frame constants $a,b$, the associated Gramian possesses the spectral gap
 $$
	\sigma_2\left(\left(\langle m_\lambda, m_{\lambda'} \rangle_{L^2(\R^2)}\right)_{\lambda,\lambda'\in \Lambda}\right)\subset \{0\}\cup[a,b].
 $$
 Furthermore, the Moore-Penrose pseudoinverse of the Gramian $(\langle m_\lambda, m_{\lambda'} \rangle_{L^2(\R^2)})_{\lambda,\lambda'\in \Lambda}$
is given by the dual Gramian $(\langle \tilde{m}_\lambda, \tilde{m}_{\lambda'} \rangle_{L^2(\R^2)})_{\lambda,\lambda'\in \Lambda}$.
\end{Lemma}
\begin{proof}
 \cite[Lemma 3.3]{Grohs2013}.
\end{proof}
In view of Lemma \ref{lemma:gram}}, if we can show that the index distance
$\omega_\alpha$ restricted to a suitable discrete index
set satisfies the assumptions of Section \ref{sec:framework},
we can directly appeal to Theorem \ref{thm_main}
to deduce localization results for the dual frame. 
The verification of these latter properties is the subject
of the remainder of this section. In particular
we shall consider curvelet-type and shearlet-type sampling sets below. 
\subsubsection{Curvelet-type Parametrization}\label{subsubsec:curve}
We start by considering curvelet-type parametrizations which
arise by discretizing the scale parameter on a logarithmic scale and the directional parameter uniformly in polar angle (see \cite{Grohs2013a} for more details).
We show the admissibility and the separatedness of the resulting parametrization, which allows us to directly appeal to Theorem \ref{thm_main}.
 \begin{Def}
    Let $\alpha\in[0,1]$ and $g>1$, $\tau>0$ be some fixed parameters.
    Further, let $(\gamma_j)_{j\in\N}$ and \new{$(L_j)_{j\in\N}$} be sequences of positive real numbers
    with  $\gamma_j\asymp g^{-j(1-\alpha)}$, i.e.\ there
    are constants $C,c>0$ independent of $j$ such that $c g^{-j(1-\alpha)}\le \gamma_j \le C g^{-j(1-\alpha)}$,
    and $ L_j \lesssim g^{j(1-\alpha)} $.
    An \emph{$\alpha$-curvelet parametrization} is given by an index set of the form
    $$
        \Lambda^c_\alpha:=\left\{ (j,l,k) \in \N \times \Z \times \Z^2 :
        |l| \le L_j \right\}
    $$
    and a mapping
    $$
      \Phi^c (\lambda) := (s_\lambda,\theta_\lambda,x_\lambda) := (g^j,l\gamma_j,R_{\theta_\lambda}^{-1}D_{s_\lambda}^{-1} \tau k) .
    $$
  \end{Def}

  The parameters $g>1$ and $\tau>0$ are sampling constants which determine the fineness of the sampling grid,
  $g$ for the scales and $\tau$ for locations.
  
  Our next goal is to prove that the index distance
  $\omega_\alpha$ which arises
  from a curvelet-type parameterization separates the index set $\Lambda^c_\alpha$ and is admissible
  as defined in Subsection \ref{subsec:basicnotions}. 
  We start by proving separatedness below. 
  \begin{Prop}
   The index set $\Lambda^c_\alpha$ is separated by $\omega_\alpha$ with
   $$ C_{\Lambda_\alpha^c} = \min\{g,1+c^2,1+\tau^2,1+\tau\} , $$
   where $ c = \inf\limits_{\lambda\in\Lambda^c_\alpha} \gamma_j g^{j(1-\alpha)} $.
  \end{Prop}
  \begin{proof}
   Let $ \lambda,\lambda' \in \Lambda $, $ \lambda \neq \lambda' $.
   If $ j \neq j' $, then
   $$ \omega(\lambda,\lambda') = g^{|j-j'|}(1 + d(\lambda,\lambda') \geq g . $$
   Thus we can suppose $ j = j' $, so that $ \omega(\lambda,\lambda') = 1 + d(\lambda,\lambda') $.
   Now, if $ l \neq l' $ we estimate
   $$ |\Delta\theta_\lambda|^2 = |l-l'|^2 \gamma_j^2 \geq c^2 |l-l'|^2 g^{-j(2-2\alpha)} \geq c^2 g^{-j(2-2\alpha)} , $$
   whence
   $$ \omega(\lambda,\lambda') \geq 1 + c^2 g^{j(2-2\alpha)} g^{-j(2-2\alpha)} = 1 + c^2 . $$
   Thus we can finally suppose $ j = j' $, $ l = l' $ and $ k \neq k' $. If $ k_2 \neq k'_2 $ we estimate
   \begin{align*}
    \|\Delta x_\lambda\|^2 &= \| R_{\theta_\lambda}^{-1} D_{s_\lambda}^{-1} \tau (k-k') \|^2 \\
                           &= \tau^2 \| D_{s_\lambda}^{-1} (k-k') \|^2 \\
                           &= \tau^2 [g^{-j2}(k_1-k'_1)^2 + g^{-j2\alpha}(k_2-k'_2)^2] \\
                           &\geq \tau^2 g^{-j2\alpha}(k_2-k'_2)^2 \\
                           &\geq \tau^2 g^{-j2\alpha} ,
   \end{align*}
   whence
   $$ \omega(\lambda,\lambda') \geq 1 + \tau^2 g^{j2\alpha} g^{-j2\alpha} = 1 + \tau^2 . $$
   Otherwise, if $ k_2 = k'_2 $ and $ k_1 \neq k'_1 $ we estimate
   $$ |\langle \Delta x_\lambda , e_{\theta_\lambda} \rangle| = | g^{-j}\tau(k_1-k'_1 ) | = g^{-j} \tau |k_1-k'_1| \geq g^{-j} \tau , $$
   whence
   \begin{equation*}
    \omega(\lambda,\lambda') \geq 1 + g^{j} g^{-j} \tau = 1 + \tau . \qedhere
   \end{equation*}
  \end{proof}
  
  Next we examine the admissibility of $\omega_\alpha$
  restricted to $\Lambda_\alpha^c$. This property has actually
  been verified in \cite{Grohs2013a} and we arrive at the following result.
  \begin{Prop}
   The index distance $\omega_\alpha$ on $\Lambda_\alpha^c$ is $2$-admissible for all $\alpha \in [0,1] $.
  \end{Prop}
  \begin{proof}
   \cite[Proposition 3.6]{Grohs2013a}.
  \end{proof}
  
  We can finally apply our machinery to obtain the following localization result for $\alpha$-curvelets.
  \begin{Thm} \label{thm:curveloc}
   Assume we have an $\alpha$-curvelet frame which is $(N\!+\!L)$-localized with respect to $\omega_\alpha$,
   with $N$ and $L$ satisfying \ref{eq:hp}.
   Then the dual frame is $N^+\!$-localized, with $N^+$ given by \ref{eq:nplus}.
  \end{Thm}
  \begin{proof}
   Just apply Theorem \ref{thm_main}\new{, taking into account Lemma \ref{lemma:gram}}.
  \end{proof}
  
  \subsubsection{Shearlet-type Parametrization}\label{subsubsec:shear}
This subsection studies the admissibility and separatedness
of so-called shearlet-type parametrizations as introduced
in \cite{Grohs2013a}. This paramatrization discretizes
the directional parameter uniformly in slope rather
than in angle, which is advantageous for digital implementations.
\new{This is done by means of the \emph{shear transformation}
\begin{equation} 
     S_t := \begin{pmatrix}
             1 & t \\
             0 & 1               
            \end{pmatrix} ,
\end{equation}
which replaces the usual rotation $R_\theta$ with $ t = \tan\theta $.}
\begin{Def}
  Let $ \alpha \in [0,1] $, $ g > 1 $ and $ \tau > 0 $.
  Further, let $(\eta_j)_{j\in\Z}$ and \new{$(L_j)_{j\in\Z}$} be sequences of positive real numbers
  with $ \eta_j \asymp g^{-j(1-\alpha)}$ and \new{$ L_j \lesssim g^{j(1-\alpha)} $}.
  An \emph{$\alpha$-shearlet parametrization} is given by an index set of the form
  $$ \Lambda^s_\alpha := \{ (j,l,k) \in \Z \times \Z \times \Z^2 : |l| \leq \new{L_j} \} $$
  and a mapping
  $$ \Phi^s (\lambda) := (s_\lambda,\theta_\lambda,x_\lambda) := (g^j,\arctan(l\eta_j),S_{\tan\theta_\lambda}^{-1}D_{s_\lambda}^{-1} \tau k) . $$
\end{Def}

\new{
\begin{Rmk}
 We may suppose to work with a scaling function and thus consider only positive scales $j\in\N$,
 but this would require additional notation and inessential slight complications on the paramatrization.
 Anyway, all the arguments go through with or without scaling functions.
\end{Rmk}
}
  
  Similarly to the $\alpha$-curvelet case, we first prove that $\omega_\alpha$ separates $\Lambda_\alpha^s$.
  \begin{Prop}
   The index set $\Lambda_\alpha^s$ is separated by $\omega_\alpha$, with
   $$ C_{\Lambda_\alpha^s} = \min\{g,1+c^2(1+C^2)^{-2},1+\tau^2,1+\tau(1+C^2)^{-1/2}\} , $$
   where $ c = \inf\limits_{\lambda\in\Lambda^s_\alpha} \eta_j g^{j(1-\alpha)} $ and
   $ C = \new{\sup\limits_{j\in\Z} L_j \eta_j} $.
  \end{Prop}
  \begin{proof}
   Let $ \lambda,\lambda' \in \Lambda $, $ \lambda \neq \lambda' $.
   If $ j \neq j' $, then
   $$ \omega(\lambda,\lambda') = g^{|j-j'|}(1 + d(\lambda,\lambda') \geq g . $$
   Thus we can suppose $ j = j' $, so that $ \omega(\lambda,\lambda') = 1 + d(\lambda,\lambda') $.
   If $ l \neq l' $ we estimate $ |\Delta\theta_\lambda|^2 $.
   By the mean value theorem we have
   $$ |\Delta\theta_\lambda|^2 = |\arctan(l\eta_j) - \arctan(l'\eta_j)|^2 = \eta_j^2 |l-l'|^2 \left( \frac{1}{1+\xi^2} \right)^2 $$
   for some $\xi$, $ |\xi| \leq |\max(l,l')|\eta_j \new{\leq L_j\eta_j} \leq C $, so that
   $$ |\Delta\theta_\lambda|^2 \geq c^2 g^{-j(2-2\alpha)} |l-l'|^2 (1+C^2)^{-2} \geq c^2 g^{-j(2-2\alpha)} (1+C^2)^{-2} , $$
   whence
   $$ \omega(\lambda,\lambda') \geq 1 + c^2 g^{j(2-2\alpha)} g^{-j(2-2\alpha)} (1+C^2)^{-2} = 1 + c^2(1+C^2)^{-2} . $$
   Thus we can finally suppose $ j = j' $, $ l = l' $ and $ k \neq k' $. If $ k_2 \neq k'_2 $ we estimate
   \begin{align*}
    \|\Delta x_\lambda\|^2 &= \| S_{l\eta_j}^{-1}D_{g^j}^{-1} \tau (k-k') \|^2 \\
                           &= \tau^2 \{[g^{-j}(k_1-k'_1) - l\eta_j g^{-j\alpha}(k_2-k'_2)]^2 + g^{-j2\alpha}(k_2-k'_2)^2\} \\
                           &\geq \tau^2 g^{-j2\alpha}(k_2-k'_2)^2 \\
                           &\geq \tau^2 g^{-j2\alpha} ,
   \end{align*}
   whence
   $$ \omega(\lambda,\lambda') \geq 1 + \tau^2 g^{j2\alpha} g^{-j2\alpha} = 1 + \tau^2 . $$
   Otherwise, if $ k_2 = k'_2 $ and $ k_1 \neq k'_1 $ we estimate
   \begin{align*}
    |\langle \Delta x_\lambda , e_{\theta_\lambda} \rangle| &= | g^{-j}\tau(k_1-k'_1 ) \cos\theta_\lambda | \\
                                                            &= g^{-j}\tau|k_1-k'_1| \cos\theta_\lambda \\
                                                            &= g^{-j}\tau|k_1-k'_1| \cos\arctan(l\eta_j) \\
                                                            &= g^{-j}\tau|k_1-k'_1| \frac{1}{\sqrt{1 + l^2 \eta_j^2}} \\
                                                            &\geq g^{-j}\tau (1+C^2)^{-1/2},
   \end{align*}
   whence
   \begin{equation*}
    \omega(\lambda,\lambda') \geq 1 + g^{j} g^{-j}\tau (1+C^2)^{-1/2} = 1 + \tau (1+C^2)^{-1/2} . \qedhere
   \end{equation*}
  \end{proof}
  
  Finally, it only remains to see the matter of admissibility. As before, we refer to \cite{Grohs2013a}.
  \begin{Prop}
   The index distance $\omega_\alpha$ on $\Lambda_\alpha^s$ is $2$-admissible for all $\alpha \in [0,1] $.
  \end{Prop}
  \begin{proof}
   \cite[Proposition 3.6]{Grohs2013a}.
  \end{proof}
  
  We conclude by stating the corresponding localization result for $\alpha$-shearlet frames.
  \begin{Thm} \label{thm:shearloc}
   Assume we have an $\alpha$-shearlet frame which is $(N\!+\!L)$-localized with respect to $\omega_\alpha$,
   with $N$ and $L$ satisfying \ref{eq:hp}.
   Then the dual frame is $N^+\!$-localized, with $N^+$ given by \ref{eq:nplus}.
  \end{Thm}
  \begin{proof}
   Just apply Theorem \ref{thm_main}\new{, taking into account Lemma \ref{lemma:gram}}.
  \end{proof}

 \appendix
 \section{Auxiliary Results}\label{appendix}
  \begin{Lemma} \label{lemma:trimin}
   For all $ t,t',t'' \in \R $ and $ \beta \leq 2 $,
   $$ |t-t'| + \beta\min(t,t') \leq |t-t''| + |t''-t'| + \beta\min(t,t',t'') ; $$
   if $ \beta \in [0,2] $, then
   $$ |t-t'| + \beta\min(t,t') \leq |t-t''| + |t''-t'| + \beta \begin{cases} \min(t,t'') \\
                                                                                  \min(t'',t')
                                                                    \end{cases} . $$
  \end{Lemma}
  \begin{proof}
   If $ t'' \geq \min(t,t') $, then $ \min(t,t') = \min(t,t',t'') $, hence $ \beta\min(t,t') = \beta\min(t,t',t'') $ for any $ \beta \in \R $.
   Thus the first inequality of the thesis reduces to the usual triangle inequality.
   But if $ t'' \leq \min(t,t') $, then $ \min(t,t') \geq \min(t,t',t'') $,
   and we need the triangle inequality to counterbalance this effect.
   Of course we can suppose $ t'' \leq t \leq t' $, so that $ \min(t,t') = t $ and $ \min(t,t',t'') = t'' $.
   The left term is now $ t' + (\beta-1) t $, while the right term is $ t' + t + (\beta-2)t'' $,
   so that the inequality is equivalent to $ (\beta-2)t \leq (\beta-2)t'' $,
   which is true for $ t'' \leq t $ if $ \beta \leq 2 $.
   
   If in addition $ \beta \geq 0 $, we have $ \beta\min(t,t',t'') \leq \beta m $,
   where $m$ is either $\min(t,t'')$ or $\min(t'',t')$,
   whence then second couple of inequalities.
  \end{proof}
  
  \begin{Lemma} \label{lemma:maxmin}
   For all $ t,t',t'' \in \R $
   $$ \max(\min(t,t''),\min(t'',t')) \leq \min(\max(t,t''),\max(t'',t')) . $$
  \end{Lemma}
  \begin{proof}
   We can suppose $ t \leq t' $ and check the left term $L$ and the right term $R$ in each case.
   \begin{description}
    \item[$ t'' \leq t \leq t' $ : ]
     $ L = t'' $, $ R = t $.
    \item[$ t \leq t'' \leq t' $ : ]
     $ L = t'' $, $ R = t'' $.
    \item[$ t \leq t' \leq t'' $ : ]
     $ L = t' $, $ R = t'' $. \qedhere
   \end{description}
  \end{proof}
\end{document}